\documentclass{amsart}
\usepackage{verbatim}
\usepackage[english]{babel}
\usepackage{amsmath}
\usepackage{amsthm}
\usepackage{enumerate}
\usepackage{graphicx}
\usepackage{url}
\usepackage[all]{xy}
\usepackage[usenames]{color}
\usepackage{calc}

\usepackage[top=1.35in, bottom=1.35in, left=1.70in, right=1.70in]{geometry}

\newtheorem{thm}{Theorem}

\newtheorem{prop}[thm]{Proposition}

\newtheorem{conj}[thm]{Conjecture}
\newtheorem*{claim}{Claim}
\newtheorem{ques}[thm]{Question}
\theoremstyle{remark}
\newtheorem{rem}[thm]{Remark}

\newcommand{\gsm}{g_{4}}
\newcommand{\gtop}{g_{4}^{{\rm top}}}
\newcommand{\rk}{{\rm rank} }
\newcommand{\Kmn}{\mathcal{K}_{m,n} }
\newcommand{\Z}{\mathbb{Z}}
\newcommand{\norm}[1]{\lVert #1 \rVert^2}

\def\labelm{\overbrace{\hspace{2.2cm}}^{2m+3\text{ crossings}}}
\def\labeln{\overbrace{\hspace{2.2cm}}^{2n+4\text{ crossings}}}
\def\et{\quad\mbox{and}\quad}
\keywords{2-bridge knots, smooth slice genus, topological slice genus}
\subjclass[2010]{57M25,  57M27}

\title{On 2-bridge knots with differing smooth and topological slice genera}
\author{Peter Feller}
\author{Duncan McCoy}
\address{Boston College, Department of Mathematics, Maloney Hall, Chestnut Hill, MA 02467, United States}
\address{School of Maths \& Statistics, University of Glasgow, 17 University Gardens, Glasgow, United Kingdom}
\date{}
\begin{document}

\begin{abstract}
We give infinitely many examples of 2-bridge knots for which the topological and smooth slice genera differ. The smallest of these is the 12-crossing knot $12a255$. These also provide the first known examples of alternating knots for which the smooth and topological genera differ.
\end{abstract}
\maketitle

\section{Introduction}
For a knot $K \subset S^3$, the smooth (resp.\ topological) slice genus, $\gsm(K)$ (resp. $\gtop(K)$), is defined to be the minimal possible genus of a properly embedded smooth (resp.\ locally flat) surface in $B^4$  with boundary $K$. The slice genera generalize the notion of sliceness in that, by definition, a knot $K$ is smoothly (resp.\ topologically) slice if and only if
$\gsm(K)$ (resp. $\gtop(K)$) is zero.

Since any smoothly embedded surface is also locally flat, we always have the inequality
\[\gtop(K) \leq \gsm(K).\]
It is a well-known phenomenon of knot theory and 4-dimensional topology that this inequality can be strict. For example, the $(-3,5,7)$-pretzel knot is topologically, but not smoothly, slice. The difference between the two genera can even be made arbitrarily large~\cite{Rudolph_84_SomeTopLocFlatSurf,KronheimerMrowka_Gaugetheoryforemb}. However, many of the invariants with which one could hope to prove $\gtop(K) < \gsm(K)$ cannot be applied when $K$ is 2-bridge. For example, when $K$ is alternating, the Rasmussen $s$-invariant and the Ozsv{\'a}th and Szab{\'o} $\tau$-invariant both give the same lower bound for $\gsm(K)$ as the signature $\sigma(K)$, and, hence, also give an identical lower bound for $\gtop(K)$~\cite{rasmussen10khovanov,Ozsvath03alternating}.
In fact, conjecturally, the topological and smooth notion of sliceness are the same for 2-bridge knots.
\begin{conj}[\cite{Miller_15_TopSlicenessOf2bridge,EisermannLamm_09,CassonGordon_86}]\label{conj:topslice=smoothslice}
For all $2$-bridge knots $K$, $\gsm(K)=0$ if and only if $\gtop(K)=0$.
\end{conj}
The aim of this paper is to show that for 2-bridge knots the smooth and topological genera are different. In other words, the natural generalization of Conjecture~\ref{conj:topslice=smoothslice} to higher genera is false.

Let $\Kmn$ be, as depicted in Figure~\ref{fig:Kmn}, the 2-bridge knot corresponding to the rational number\footnote{{Here the 2-bridge knot corresponding to $p/q$ means the knot whose double branched cover is the lens space $L(p,q)$. By taking a continued fraction $p/q=[a_0, \dots, a_n]^+$ in which all the coefficients are positive, one can obtain an alternating 2-bridge diagram for such a knot in which the $a_i$ count the number of crossings in each twist region.}}
\[
[2m+3,1,2n+4,1,1,2]^+= \frac{20mn+56m+40n+107}{10n+28}.
\]

\begin{figure}[h]
\centering
\def\svgscale{2.5}
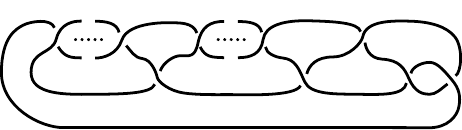
\caption{\label{fig:Kmn} The knot $\Kmn$}
\end{figure}
The main result of this paper is that for many choices of $m$ and $n$, the $\Kmn$ give examples of 2-bridge knots with differing topological and slice genera. To the best of the authors' knowledge, these are also the first known instances of alternating knots for which the smooth and topological genera differ.

\begin{thm}\label{thm:main}
For all $m,n\geq 0$, the smooth slice genus of $\Kmn$ is $\gsm(\Kmn)=2$. Furthermore, if any of the following hold:
\begin{enumerate}[(i)]
\item $m=n=0$,
\item $m+3$ is a perfect square or
\item $n+2$ is a perfect square,
\end{enumerate}
then $\gtop(\Kmn)=1$.
\end{thm}

Given that all our examples satisfy $\gsm(K)-\gtop(K)\leq 1$, it seems natural to ask the following.
\begin{ques}\label{ques:differences}
Are there 2-bridge knots for which the difference $\gsm(K)-\gtop(K)$ is arbitrarily large?
\end{ques}
Although it is seems likely that this has an affirmative answer, proving this would require exhibiting 2-bridge knots for which $2\gsm(K)-|\sigma(K)|$ is arbitrarily large. In particular, addressing Question~\ref{ques:differences} would require an approach to the smooth slice genus differing from the one taken in this paper: the best our method can achieve is to show $2\gsm(K)-|\sigma(K)|\ne 0$.
\begin{rem}
The knot $\mathcal{K}_{0,0}$ appears in the tables as $12a255$ \cite{Knotinfo}. With the possible exception of $11a211$, for which $\gsm(11a211)=2$, but the authors are unable to determine the topological slice genus, there are no alternating knots with 11 or fewer crossings for which the smooth and topological slice genus differ \cite{Mccoy15calculating}.
\end{rem}
\begin{rem}
Consider $\Kmn$ such that $\gtop(\Kmn)=1$. Since $\sigma(\Kmn)=-2$, the double branched cover $\Sigma(\Kmn)$ is a lens space bounding a negative-definite topological manifold with $b_2=2$. However, the proof of Theorem~\ref{thm:main} implies that this 4-manifold cannot be smooth. In fact, it follows from the proof of Proposition~\ref{prop:smoothg} that if $m>0$, then $\Sigma(\Kmn)$ cannot bound a smooth negative-definite manifold with $b_2<4$.
\end{rem}

\subsection*{Acknowledgments}

The first author gratefully acknowledges support by the Swiss National Science Foundation Grant 155477.
This work began whilst the second author was visiting Boston College. He would like to thank Josh Greene for his hospitality throughout his stay. He is also grateful to the College of Science and Engineering, University of Glasgow, for providing the funding that made his visit possible. The authors would also like to thank the anonymous referee for their comments and suggestions.

\section{Proof of Theorem~\ref{thm:main}}
In this section, we prove the main result of this paper. Since $\Kmn$ possesses a genus two Seifert surface, as shown in Figure~\ref{fig:Knm}, and has signature $\sigma(\Kmn)=-2$, we always have
\[1\leq\gtop(\Kmn)\leq \gsm(\Kmn)\leq 2.\]
The statements concerning the topological genus in Theorem~\ref{thm:main} are established by using Freedman's Disk theorem to construct a locally flat genus one surface bounding $\Kmn$. The smooth part of Theorem~\ref{thm:main} is established by using Donaldson's diagonalisation theorem to show $2\gsm(\Kmn)\ne |\sigma(\Kmn)|$. These steps are carried out in the following two sections.

\subsection{The topological genus}
\begin{prop}\label{prop:topsliceKnm=1}
If $n=m=0$, $m+3$ is a perfect square, or $n+2$ is a perfect square, then
the topological slice genus of $\Kmn$ is 1.
\end{prop}
{The strategy of proof for Proposition~\ref{prop:topsliceKnm=1} is to find an essential separating curve $L$ on a minimal genus Seifert surface $S$ for $K$ with trivial Alexander polynomial. This allows us to establish that the topological slice genus is strictly smaller than the genus of $S$ by invoking Freedman's Disk Theorem~\cite{Freedman_82_TheTopOfFour-dimensionalManifolds}.}

This idea was first used by Rudolph to show that the $\gtop(K)<\gsm(K)$ for most torus knots~\cite{Rudolph_84_SomeTopLocFlatSurf}. The curve $L$ is found by {considering the} Seifert form; compare also~\cite
{Feller_15_DegAlexUpperBoundTopSliceGenus} and~\cite{BaaderLewark_15_Stab4GenusOFAltKnots} for this type of argument.

\begin{proof}[Proof of Proposition~\ref{prop:topsliceKnm=1}]
We establish Proposition~\ref{prop:topsliceKnm=1} by showing that $\Kmn$ has a genus two Seifert surface $S$ {containing a} simple closed curve $L$ such that $\Delta_L=1$ and $L$ separates $S$ into two connected components of genus one. Here, $\Delta_L$ denotes the Alexander polynomial of a knot $L\subset S^3$. Indeed, this suffices to show $\gtop(\Kmn)\leq 1$ by the following argument:
removing the connected component $C$ of $S\setminus L$ that does not contain $\Kmn$, one obtains a genus one surface $\overline{S\setminus C}\subset S^3$ with boundary $\Kmn\sqcup L$.
By Freedman's {D}isk Theorem~\cite[Theorem~1.13]{Freedman_82_TheTopOfFour-dimensionalManifolds}, $L$ bounds a topological locally flat disc $B^2$ in $B^4$.
Gluing $\overline{S\setminus C}$ and $B^2$ along $L$ yields a locally flat surface in $B^4$ of genus $1$ with boundary $\Kmn\subset S^3=\partial B^4$.

We orient $\Kmn$ and choose $S$ to be the genus two Seifert surface for $\Kmn$ depicted in Figure~\ref{fig:Knm}.
\begin{figure}[h]
\centering
\def\svgscale{2.5}
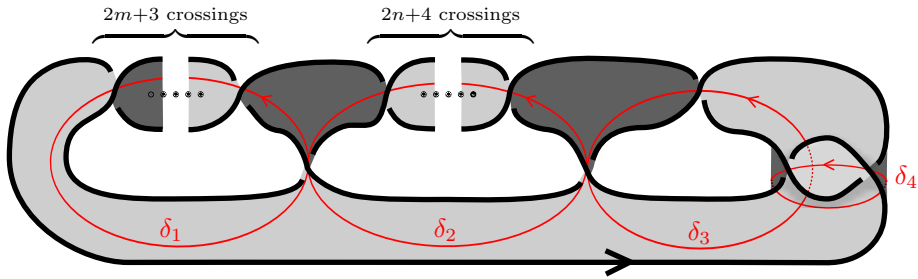
\caption{A Seifert surface $S$ (gray) for the knot $\Kmn$ (black) with $4$ simple closed curves 
(red) such that 
their homology classes provide a basis for $H_1(S,\Z)$. The orientation of $S$ is chosen such that its boundary induces the indicated orientation of $\Kmn$.}
\label{fig:Knm}
\end{figure}

A Seifert matrix for $S$ is
\begin{equation}\label{eq:Seifertmatrix}
M=\left[\begin{array}{cccc}
-m-2&1&0&0\\
0&-n-3&1&0\\
0&0&-1&0\\
0&0&-1&1
\end{array}\right].\end{equation}
Indeed, $M$ is the matrix with entries
\[m_{ij}={\rm lk}(\delta_i,\delta_j^+),\]
where the $\delta_i$ are the curves indicated in Figure~\ref{fig:Knm},
$\delta_j^+$ denotes the curve in $S^3\setminus S$ obtained by
a small perturbation of $\delta_j$ in the positive normal direction to
$S$,
and ${\rm lk}(\cdot,\cdot)$ denotes the linking number between disjoint curves in $S^3$.
In other words, $M$ is the Seifert matrix obtained by writing the Seifert form with respect to the basis 
of $H_1(S,\Z)$ given by $([\delta_1],[\delta_2],[\delta_3],[\delta_4])$.
Note that the signature of the symmetrization of $M$ is $-2$, which establishes $\sigma(\Kmn)=-2$ as claimed earlier.

Next we pick two elements $a$ and $b$ in $H_1(S,\Z)$ such that the Seifert form restricted to the span of $a$ and $b$ is the same as the Seifert form on a genus one Seifert surface of the unknot.
We first consider the case where $m+2$ is a perfect square. The Seifert form on a genus one Seifert surface of the unknot is (for an appropriately chosen basis) given by a Seifert matrix of the form
$\left[\begin{array}{cc}0&1\\0&*\end{array}\right]$. Therefore, we look for elements $a\in H_1(S,\Z)$ such that ${\rm lk}(a,a^+)=0$. Inspection of $M$ reveals that $a={[\delta_1] +\sqrt{m+2}[\delta_4]}$ is such an element.
Setting 
$b=[\delta_2]$, yields
\begin{equation}\label{eq:SeifertForme1e2form+2}\left[\begin{array}{cc}{\rm lk}(a,a^+)&{\rm lk}(a,b^+)\\{\rm lk}(b,a^+)&{\rm lk}(b,b^+)\end{array}\right]=\left[\begin{array}{cc}0&1\\0&-n-3\end{array}\right]\end{equation}
as wanted. Similar choices for $a$ and $b$ are possible in the cases where $n+3$ is a perfect square or $n=m=0$.
Concretely, if we choose $a$ and $b$ in $H_1(S,\Z)$ as 
\begin{align*}a&=[\delta_1],&b=[\delta_2] +\sqrt{n+3}[\delta_4]\quad\quad&\text{if $n+3$ is a square}\quad &\et\\
a&=[\delta_1]+[\delta_4],&b=[\delta_1]+[\delta_2]+ 2[\delta_4]\quad\quad&\text{if $m=n=0$}&,\end{align*} respectively;
then a calculation shows that we have
\begin{equation}\label{eq:SeifertForme1e2}\left[\begin{array}{cc}{\rm lk}(a,a^+)&{\rm lk}(a,b^+)\\{\rm lk}(b,a^+)&{\rm lk}(b,b^+)\end{array}\right]=\left[\begin{array}{cc}*&1\\0&0\end{array}\right].\end{equation}

Our choices of $a$ and $b$ ensures that $a$ intersects $b$ once algebraically on $S$ (for this recall that the algebraic intersection form on $H_1(S,\Z)$ is equal to the antisymmetrization of the Seifert form).
Therefore, the classes $a$ and $b$ are represented by simple closed curves $\alpha$ and $\beta$ in $S$ that intersect once transversally; see e.g.~\cite[Third proof of Theorem~6.4]{FarbMargalit_12_APrimerOnMCG}.
 Let $C$ be a closed neighborhood of the union of $\alpha$ and $\beta$ and let $L$ be the boundary of $C$.
By construction, $L$  separates $S\setminus L$ into two surfaces of genus $1$.
Furthermore, $\Delta_L=1$ since the surface $C$ is a Seifert surface for $L$ and, therefore, \eqref{eq:SeifertForme1e2form+2} respectively \eqref{eq:SeifertForme1e2} is a Seifert matrix for $L$.
\end{proof}
\begin{rem}
The above proof yields the curve $L$ explicitly. Indeed, the only non-explicit step in the above proof\textemdash the realization of the homology classes $a$ and  $b$ by once intersecting simple closed curves $\alpha$ and $\beta$\textemdash can be accomplished by a geometric version of the Euclidian algorithm; compare~\cite[Third proof of Theorem~6.4]{FarbMargalit_12_APrimerOnMCG}.
\end{rem}

\subsection{The smooth genus}
Let $K$ be a knot with reduced alternating diagram $D$. Let us take a chessboard colouring of this diagram such that every crossing has incidence number $\mu=-1$, as according to the convention in Figure~\ref{fig:incidence}.
\begin{figure}[h]
\centering
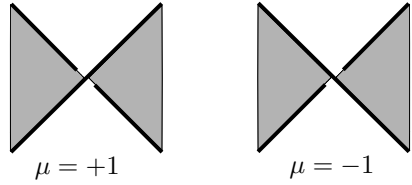
\caption{The incidence number of a crossing.}
\label{fig:incidence}
\end{figure}
If we label the white regions $v_1, \dots, v_{r+1}$, then we can define the Goeritz form $G_D$ on $\oplus_{i=1}^r v_i \Z$ by
\[
v_{i}\cdot v_j =
\begin{cases}
\text{\# crossings around $v_i$ if $i=j$ and}\\
-\text{\# crossings between $v_i$ and $v_j$ if $i\ne j$.}
\end{cases}
\]
As $D$ is a reduced alternating diagram, this is positive definite. The signature of $K$ can be computed by the following formula \cite{gordon1978signature}:
\begin{equation}\label{eq:signature}
\sigma(K)= \rk(G_D) - n_+,
\end{equation}
where $n_+$ denotes the number of positive crossings in $D$. We will use the following obstruction, obtained by combining the work of Gordon and Litherland with Donaldson's theorem, to determine $\gsm(\Kmn)$.
\begin{prop}\label{prop:obstruction}
Let $K$ be a knot with a reduced alternating diagram $D$ and $\sigma(K)\leq 0$. If $\gsm(K)=-\sigma(K)/2$, then there is a lattice embedding
\[G_D \hookrightarrow \Z^{\rk(G_D)-\sigma(K)}.\]
\end{prop}
\begin{proof}
Let $X$ be the 4-manifold obtained by taking a double cover of $B^4$ branched over the (not necessarily orientable) surface given by the shaded regions of $D$. This has intersection form isomorphic to $G_D$ and boundary $\Sigma(K)$ \cite{gordon1978signature}. On the other hand, taking a double cover of $B^4$ branched over a minimal genus smooth surface with boundary $K$ gives a 4-manifold $Y$ with signature $\sigma(K)$ and $b_2(Y)=2\gsm(K)$ \cite{gordon1978signature,Kauffman76signature}. Consider the smooth closed 4-manifold $Z:=X\cup_{\Sigma(K)} -Y$. If $\gsm(K)=-\sigma(K)/2$, then $Z$ is positive definite and $b_2(Z)=\rk(G_D)-\sigma(K)$. Since this implies that the intersection form of $Z$ must be diagonalisable \cite{donaldson87orientation}, the inclusion $X\subset Z$ induces the required embedding of lattices.
\end{proof}
\begin{rem}
A good knot with which to test that Proposition~\ref{prop:obstruction} is compatible with the shading conventions in use is the positive clasp knot $5_2$. This has signature -2 according to \eqref{eq:signature}, and $\gsm(5_2)=g(5_2)=1$. It can easily be verified that its positive Goeritz form, which has rank 3, embeds into $\Z^5$, but the positive Goeritz form for its mirror, which has rank 2, does not embed into $\Z^4$.
\end{rem}
For $\Kmn$, the Goeritz form is isomorphic to $Q_{m,n}$, the lattice generated by vectors $v_1, \dots , v_{2m+2n+8}$, with pairing given by
 \[v_i\cdot v_j =
 \begin{cases}
 3 & \text{if } i=j\in \{2m+3, 2n+2m+7, 2n+2m+8\}\\
 2 & \text{if } i=j\not\in \{2m+3, 2n+2m+7, 2n+2m+8\}\\
 -1 & \text{if } |i-j|=1\\
 0 & \text{if } |i-j|>1.
 \end{cases}
 \]
In fact, $Q_{m,n}$ is isomorphic to the intersection form of the positive definite plumbing $P_{m,n}$, as shown in Figure~\ref{fig:plumbing}. In this case, the manifold $X$ arising in the proof of Proposition~\ref{prop:obstruction} is in fact diffeomorphic to $P_{m,n}$.

\begin{figure}[h]
\[\xymatrix{
 *\txt{2\\$\bullet$\\ \makebox[2pt]{} } \ar@{--}[r]_{\underbrace{\text{\makebox[1.1cm]{}}}_{2m+2}} &
  *\txt{2\\ $\bullet$\\ \makebox[2pt]{}} \ar@{-}[r] &
  *\txt{3 \\ $\bullet$\\ \makebox[2pt]{}} \ar@{-}[r] &
  *\txt{2 \\ $\bullet$\\ \makebox[2pt]{}} \ar@{--}[r]_{\underbrace{\text{\makebox[1.1cm]{}}}_{2n+3}}&
   *\txt{2 \\ $\bullet$\\ \makebox[2pt]{}} \ar@{-}[r] &
   *\txt{3 \\ $\bullet$\\ \makebox[2pt]{}} \ar@{-}[r]&
    *\txt{3\\ $\bullet$\\ \makebox[2pt]{}}}
 \]
\caption{\label{fig:plumbing} The plumbing $P_{m,n}$}
\end{figure}
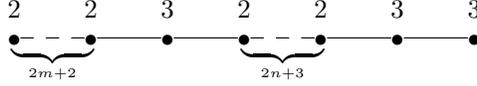

The following proposition completes the proof of Theorem~\ref{thm:main}.
\begin{prop}\label{prop:smoothg}
For all $m,n\geq 0$, the smooth genus of $\Kmn$ is $\gsm(\Kmn)=2$.
\end{prop}
\begin{proof}
It is easy to verify, either by using \eqref{eq:signature} or computing directly from the Seifert matrix \eqref{eq:Seifertmatrix}, that $\sigma(\Kmn)=-2$ for all $m,n\geq 0$. Since $\Kmn$ has a genus two Seifert surface, this shows $1\leq \gsm(\Kmn)\leq 2$. Thus the proposition follows from Proposition~\ref{prop:obstruction} combined with the following claim.
\begin{claim} Suppose that we have an embedding $Q_{m,n}\hookrightarrow \mathbb{Z}^M$, then
\[M \geq
\begin{cases}
\rk (Q_{m,n})+3 &\text{if } m=0\\
\rk (Q_{m,n})+4 &\text{if } m\geq 1.
\end{cases}\]
\end{claim}
\begin{proof}[Proof of Claim]
Suppose that $\mathbb{Z}^M$ contains $Q_{m,n}$ as a sublattice. 
Throughout this proof, we will use $e_i$ and $f_j$ to denote unit vectors in $\Z^M$ with the assumption that $e_i \cdot e_j = f_i \cdot f_j=0$ whenever $i \ne j$.

First consider the vectors of norm 2 in $\Z^M$ given by $v_1, \dots , v_{2m+2}$. Since $v_1$ and $v_2$ both have norm 2 and $v_1\cdot v_2=-1$, there must be unit vectors $e_1, e_2, e_3 \in \Z^M$ such that $v_1=e_1-e_2$ and $v_2=e_2-e_3$.
Suppose now that $m\geq 1$. In this case, $v_3$ also has $\norm{v_3}=2$, and additionally satisfies $v_3\cdot v_2=-1$ and $v_3\cdot v_1=0$. Thus we see that either (1) $v_3=-e_1-e_2$ or (2) there is $e_4 \in \Z^M$ such that $v_3=e_3-e_4$. However, if (1) holds, then any vector $x$ satisfying $x\cdot v_1=0$, must also satisfy $x \cdot v_3 = 2x\cdot e_1 \equiv 0 \bmod 2$, meaning that we obtain a contradiction when we consider $v_4$. Thus we must have $v_3=e_3-e_4$. Continuing in this way, we see that there are $e_1, \dots, e_{2m+3} \in \Z^M$ such that $v_i=e_i-e_{i+1}$ for $i=1, \dots , 2n+2$.

The vectors $v_{2m+4}, \dots, v_{2m+2n+6}$ also form a chain of vectors of norm 2 in $\Z^M$, so by an identical argument to the preceding paragraph, there must exist $f_1, \dots, f_{2n+4} \in \Z^M$ such that for $i=1, \dots 2n+3$, we have $v_{2m+3+i}=f_i-f_{i+1}$. Furthermore, since $v_i \cdot v_j=0$ for all $i<2m+3$ and $j>2m+3$, these must satisfy $e_i\cdot f_j=0$ for all $i$ and all $j$.

Now consider $v_{2m+2n+7}$. Since $v_{2m+2n+7}\cdot v_i=0$ for $i\leq 2m+2n+5$ and $v_{2m+2n+7}\cdot v_{2m+2n+6}=-1$, either $(1)$ $n=0$ and $v_{2m+2n+7}=-(f_1+f_2+f_3)$ or $(2)$ $v_{2m+2n+7}=f_{2n+4}+f_{2n+5}+f_{2n+6}$, where $f_{2n+5}$ and $f_{2n+6}$ are orthogonal to the $e_i$. Since there is no vector $v_{2m+2n+8}$ of norm 3 with $v_{2m+2n+8}\cdot (f_1+f_2+f_3)=1$ and $v_{2m+2n+8}\cdot v_i=0$ for $i\leq 2m+2n+6$, it follows that $(2)$ must always hold and $v_{2m+2n+7}=f_{2n+4}+f_{2n+5}+f_{2n+6}$. By similar considerations, we see that there must be $f_{2n+7}$ and $f_{2n+8}$ which are orthogonal to the $e_i$ and such that $v_{2m+2n+8}=-f_{2n+5}+f_{2n+7}+f_{2n+8}$ or $v_{2m+2n+8}=-f_{2n+6}+f_{2n+7}+f_{2n+8}$.

The set $\{e_1, \dots, e_{2m+3},f_1,\dots, f_{2n+8}\}$ gives a set of $2n+2m+11$ linearly independent unit vectors in $\Z^M$, showing that $M\geq 2n+2m+11 =\rk( Q_{m,n}) + 3$. In fact, if $m=0$, then $\Z^{2m+2n+11}$ contains a sublattice isomorphic to $Q_{m,n}$ (we can take $v_{3}=e_1+e_2-f_1$).

It remains only to consider $v_{2m+3}$ when $m\geq 1$. Similarly, we can deduce that $v_{2m+3}$ takes the form  $v_{2m+3}=e_{2m+3}+e_{2m+4} -f_1$, where $e_{2m+4}\in \Z^M$ is a unit vector which is orthogonal to the $f_i$. This shows that if $m\geq 1$, then there are at least $2m+2n+12$ linearly independent unit vectors in $\Z^M$. Thus we can conclude that $M\geq \rk(Q_{m,n}) + 4$, completing the proof of the claim.
\end{proof}
This completes the proof of the proposition.
\end{proof}
\bibliographystyle{alpha}
\bibliography{2bridge}
\end{document}